\documentclass[12pt,reqno]{amsart}
\usepackage{fullpage}
\usepackage{times}
\usepackage{colonequals}
\usepackage{longtable} 
\usepackage{adjustbox} 
\usepackage{enumerate}
\usepackage{colonequals}
\usepackage{adjustbox}
\usepackage[shortlabels]{enumitem}
\usepackage{amssymb}
\usepackage[linesnumbered,ruled,vlined]{algorithm2e}
\usepackage{xcolor}
\usepackage[colorlinks=true, pdfstartview=FitH, linkcolor=blue, citecolor=blue, urlcolor=blue]{hyperref}
\newenvironment{poc}{\begin{proof}[Proof of claim]}{\end{proof}}
\usepackage{comment}

\newtheorem{theorem}{Theorem}[section]

\newtheorem{lemma}[theorem]{Lemma}
\newtheorem{corollary}[theorem]{Corollary}

\newtheorem{claim}[theorem]{Claim}
\newtheorem{question}[theorem]{Question}
\theoremstyle{definition}

\newtheorem{definition}[theorem]{Definition}
\newtheorem{remark}[theorem]{Remark}
\newtheorem*{remark*}{Remark}
\numberwithin{equation}{section}

\newcommand{\bP}{\operatorname{\mathbb{P}}}
\newcommand{\F}{\operatorname{\mathbb{F}}}

\allowdisplaybreaks

\begin{document}

\title{Blocking sets from a union of plane curves}

\author{Shamil Asgarli}
\address{Department of Mathematics \& Computer Science \\ Santa Clara University \\ CA 95053, United States} \email{sasgarli@scu.edu}

\author{Dragos Ghioca}
\address{Department of Mathematics \\ University of British Columbia \\ Vancouver, BC V6T 1Z2\\ Canada}
\email{dghioca@math.ubc.ca}

\author{Chi Hoi Yip}
\address{Department of Mathematics, Hong Kong University of Science and Technology, Clear Water Bay, Hong Kong}
\email{machyip@ust.hk}


\begin{abstract} 
Motivated by a question of Erd\H{o}s on blocking sets in a projective plane that intersect every line only a few times, several authors have used unions of algebraic curves to construct such sets in $\mathbb{P}^2(\mathbb{F}_q)$. In this paper, we provide new constructions of blocking sets in $\mathbb{P}^2(\mathbb{F}_q)$ from unions of geometrically irreducible curves of a fixed degree $d$. We also establish lower bounds on the number of such curves required to form a blocking set. Our proofs combine tools from arithmetic geometry and combinatorics.
\end{abstract}

\maketitle

\section{Introduction}

Throughout the paper, $p$ denotes a prime, $q$ denotes a power of $p$, $\F_q$ denotes the finite field with $q$ elements, and $\overline{\F_q}$ denotes the algebraic closure of $\F_q$. We call an $\F_q$-rational point of $\bP^2$ an \emph{$\F_q$-point}; thus, the set of all $\F_q$-points is $\bP^2(\F_q)$. Similarly, an \emph{$\F_q$-line} in $\bP^2$ means a line whose equation is defined over $\F_q$. A set of points $B\subseteq \bP^2(\F_q)$ is a \emph{blocking set} if it intersects every $\F_q$-line. A blocking set is called \emph{trivial} if it contains all $q+1$ points of an $\mathbb{F}_q$-line; otherwise, it is \emph{nontrivial}. The smallest trivial blocking sets are lines themselves, consisting of $q+1$ points. 

The study of blocking sets is a central topic in finite geometry and design theory \cite{BSS2012, PS12}. In this paper, we study blocking sets arising from a union of plane curves. Our motivation comes from the literature on a question of Erd\H{o}s, as presented in Erd\H{o}s--Silverman--Stein \cite{ESS83}, and from our previous work \cite{AGY23}, which treated the case of a single irreducible plane curve. 

Inspired by questions related to intersection properties of set families, Erd\H{o}s--Silverman--Stein \cite{ESS83} considered the case in which the set family consists of the lines in a projective plane; this naturally led them to study blocking sets that meet each line only a few times. More precisely, let $k=k(n)$ denote the least positive integer such that in any projective plane of order $n$, there exists a blocking set $B$ of points such that $|B\cap L| < k$ for every line $L$. Erd\H{o}s asked whether $k(n)$ is bounded by a universal constant. Using the probabilistic method, Erd\H{o}s--Silverman--Stein \cite{ESS83} showed that for any constant $c>2e$, we have $k(n)<c\log n$ for all sufficiently large $n$. In the same paper, they also provided a constructive approach to show $k(n)<n-c'\sqrt{n}$ for some absolute constant $c'>0$. For extensions of the methods in \cite{ESS83} to constructions of blocking sets and related objects in finite geometry, see the surveys by Sz\H{o}nyi~\cite{Szo92b} and G\'acs--Sz\H{o}nyi~\cite{GS08}.

Erd\H{o}s' question is still wide open, and much of the subsequent work has focused on the case of a projective plane $\mathbb{P}^2(\mathbb{F}_q)$ over a finite field of order $q$. In this setting, the algebraic and geometric structure of projective Galois planes provides tools unavailable for general projective planes. In particular, it is natural to look for constructions of blocking sets arising from algebraic curves; see, for example, \cite{AL85, Szo92b, Szo92a, Ugh88}. This is partly because the number of intersections between irreducible plane curves and lines is controlled by B\'ezout's theorem. We introduce the following definition to formalize the construction of blocking sets from collections of plane curves.

\begin{definition}\label{def:blocking-family}
A blocking set $B$ in $\mathbb{P}^2(\mathbb{F}_q)$ is \emph{constructed from a union of plane curves} if $B=\bigcup_{i=1}^{\ell} C_i(\mathbb{F}_q)$ for some plane curves $C_1, \dots, C_{\ell}$ defined over $\mathbb{F}_q$, where $C_i$ is \emph{geometrically irreducible} (that is, irreducible over $\overline{\F_q}$) and has degree $d_i=\deg(C_i)>1$ for each $i$. We also say that $C_1, \dots, C_{\ell}$ form a \emph{blocking family of degree $(d_1, \dots, d_{\ell})$}.
\end{definition}

\begin{remark*}
In Definition~\ref{def:blocking-family}, we may assume without loss of generality that the curves
$C_1,\ldots,C_\ell$ are distinct. Indeed, repeated curves may simply be
discarded without changing the union $B$.
\end{remark*}

The hypothesis $d_i>1$ in Definition~\ref{def:blocking-family} is necessary to produce \emph{nontrivial} blocking sets. The geometric irreducibility condition is also natural, as one could otherwise replace a reducible curve with its irreducible components. Moreover, if $C_i$ were irreducible over $\mathbb{F}_q$, but not geometrically irreducible, a standard application of B\'ezout's theorem shows that $|C_i(\mathbb{F}_q)|\leq \frac{d_i^2}{4}$ (see, for example, \cite[Lemma 3.1]{AGY23}); consequently, $C_i(\mathbb{F}_q)$ blocks at most $\frac{d_i^2}{4}(q+1)$ lines (which is not efficient for constructing blocking sets). In contrast, geometrically irreducible curves have $q + O(\sqrt{q})$ points by the Hasse--Weil bound (see \cite{AP96} for a version that applies to singular curves).

 In our previous work \cite{AGY23-FFA-pencils, AGY23, AGY22b, AGY26}, we studied blocking sets arising from points of an irreducible plane curve, which corresponds to the case $\ell=1$. In particular, we showed in \cite{AGY23} that irreducible blocking curves of low degree $d\geq 2$ (specifically, $d < q^{1/6}$) do not exist; we also provided various constructions of irreducible blocking curves.

Let $B$ be a blocking set constructed from $C_1, \dots, C_{\ell}$ as in Definition~\ref{def:blocking-family}. By B\'ezout's theorem, we have $|B\cap L|\leq \sum_{i=1}^{\ell} \deg(C_i)$ for each $\mathbb{F}_q$-line $L$. This implies $k\leq 1+\sum_{i=1}^{\ell} \deg(C_i)$ for Erd\H{o}s' question, though this bound is often not sharp. Next, we discuss previous work using such constructions and state our new contributions.

Let $q$ be an odd prime power. Abbott and Liu \cite{AL85} constructed a blocking set in $\mathbb{P}^2(\mathbb{F}_q)$ from a union of around $\log_2 q$ conics. For any constant $c>2/\log 2$, their construction produces a set with $k< c\log q$ in Erd\H{o}s' question (while their bound $c>2/\log 2$ improves $c>2e$ from Erd\H{o}s--Silverman--Stein \cite{ESS83}, it is specific to projective Galois planes). A similar construction was independently discovered by Ughi~\cite{Ugh88}. She also proved that the number of nonsingular (equivalently, geometrically irreducible) conics required to form a blocking set must tend to infinity as $q \to \infty$, using an earlier result of hers on solutions of algebraic systems over finite fields \cite[Proposizione~B]{Ugh86}. In related work, Sz\H{o}nyi~\cite{Szo92a} constructed minimal blocking sets in $\mathbb{P}^2(\mathbb{F}_q)$ from subfamilies $\{C_a:a\in A\}$ of a pencil of conics parametrized by $a\in\F_q$. The relevant choices of $A$ are exactly the maximal independent sets in the \emph{Paley graph} on $\F_q$. Here $q\equiv 1\pmod 4$, and this graph joins distinct $a,b\in\F_q$ precisely when $a-b$ is a nonzero square in $\F_q$. With this definition, $A$ is independent if and only if every nonzero pairwise difference $a-b$, with $a,b\in A$, is a nonsquare. In the same paper, he also deduced that at least $c\log q$ conics from this specific pencil $\{C_{a}\}_{a\in \mathbb{F}_q}$ are required to form a blocking set. Our first result shows that this logarithmic lower bound holds for a blocking set formed by \emph{any} collection of geometrically irreducible conics. 

\begin{theorem}\label{thm:conics} Let $q$ be an odd prime power. There is a constant $c_0>0$ such that no blocking set in $\mathbb{P}^2(\mathbb{F}_q)$ can be constructed from a union of fewer than $c_0\log q$ geometrically irreducible conics.
\end{theorem}

For an odd prime power $q$, let $f(q)$ be the minimum integer $\ell$ such that there is a blocking set in $\bP^2(\F_q)$ constructed from $\ell$ conics. By the discussion above, we know $c_1\log q\leq f(q)\leq c_2\log q$ for some absolute constants $c_1,c_2>0$. Determining an asymptotically sharp bound on $f(q)$ seems out of reach.

As we were finalizing the paper, we discovered that Sz\H{o}nyi~\cite{Szo92b} proved an analogous lower bound for blocking sets in inversive planes constructed from a union of circles. In the final remark of the same paper, he mentioned that the same proof idea applies to conics in $\mathbb{P}^2(\mathbb{F}_q)$, so we believe that Theorem~\ref{thm:conics} is known to some experts. We will present our own proof of Theorem~\ref{thm:conics} in Section~\ref{sec:conics}, and briefly explain Sz\H{o}nyi's suggested proof in Remark~\ref{rem:S}. 

The hypothesis that $q$ is odd in Theorem~\ref{thm:conics} is necessary. Indeed, Ill\'es, Sz\H{o}nyi, and Wettl \cite{ISF91} showed that for Erd\H{o}s' question in $\mathbb{P}^2(\F_{2^r})$, the bound $k\leq 6$ holds if $r$ is even, and $k\leq 7$ holds if $r$ is odd. Their construction still employs a union of nonsingular conics.

For curves of higher degree, our next result generalizes Ughi's result \cite[Proposition 2]{Ugh88} on the number of geometrically irreducible conics needed to form a blocking set in odd characteristic.

\begin{theorem}\label{thm:intro-unbounded}
Let $d\geq 3$. Let $\ell(q)$ be the minimum integer $\ell$ such that there exists a blocking set in $\mathbb{P}^2(\mathbb{F}_q)$ constructed from a union of $\ell$ plane curves each having degree in $\{2, 3, \dots, d\}$. Let $\mathcal{Q}_d$ be the set of prime powers $q$ such that $p=\operatorname{char}(\mathbb{F}_q)>d$. Then for $q \in \mathcal{Q}_d$,  we have $\ell(q) \to \infty$ as $q\to\infty$.
\end{theorem}

We prove Theorem~\ref{thm:intro-unbounded} in Section~\ref{sec:CDT}. The key ingredient of the proof is a version of Chebotarev density theorem due to Entin \cite{Ent21}. 

Chebotarev-type arguments have recently entered finite geometry through the study of complete arcs and linear sets. Bartoli and Micheli~\cite{BM22} used a function-field Chebotarev theorem to produce lines through an external point whose intersections with a curve split completely over $\F_q$, yielding complete $m$-arcs; this framework was developed further by Bastioni and Micheli~\cite{BasMic24}. Related Galois-theoretic methods appear in the completeness problem for $(k,n)$-arcs arising from the Hermitian and BKS curves~\cite{KNS24}. The Chebotarev theorem has also been used in the study of linearized polynomials and their associated linear sets~\cite{BMZZ22}. In each case, the density theorem counts the fibers (over $\mathbb{F}_q$-points) with prescribed behavior. In our setting, we use Entin's theorem by choosing a fixed-point-free Frobenius class; this yields $\F_q$-lines $L$ satisfying $(L\cap C)(\F_q)=\emptyset$ for sufficiently large $q$; see Theorem~\ref{thm:entin} and equation~\eqref{eq:lbC} in Section~\ref{sec:CDT}.

The hypothesis $p=\operatorname{char}(\mathbb{F}_q)>d$ in Theorem~\ref{thm:intro-unbounded} is necessary. Indeed, Bruen and Fisher \cite{BF74} found a blocking set in $\mathbb{P}^2(\F_{3^r})$ formed by taking a union of geometrically irreducible cubic curves and showing that $k\leq 5$ for Erd\H{o}s' question. More generally, Boros \cite{Bor88} proved that if $q=p^{r}$ with a prime $p\geq 3$, then $k\leq p+2$ for Erd\H{o}s' question; this was achieved by considering the union of two carefully chosen geometrically irreducible degree $p$ curves together with a single point. Thus, the hypothesis $p>d$ in Theorem~\ref{thm:intro-unbounded} is sharp.

Our final result is a counterpart to these lower bounds by demonstrating that a blocking family can indeed be formed from approximately $c_d \log q$ geometrically irreducible curves of degree $d$. 

\begin{theorem}\label{thm:main}
Let $d\geq 3$. There exists a constant $c_d>0$ such that for any prime power $q$ and any integer $\ell \geq c_d \log q$, there exists a blocking family of degree $\underbrace{(d, d, \dots, d)}_{\ell \text{ times}}$ over $\mathbb{F}_q$.
\end{theorem}

We will give two proofs of Theorem~\ref{thm:main}, one probabilistic in Section~\ref{sec:random} and one through explicit equations in Section~\ref{sec:alg}. In the first proof, we use a randomized construction which produces a better constant, namely, $c_d=4-o(1)$ as $q\to \infty$. The randomized construction is also more flexible, as it can produce multiple blocking sets; see Remark~\ref{rem:multiple}. The more precise statement for the second proof appears as Theorem~\ref{thm:construction}, which gives $c_d=O(d)$. On the other hand, the second proof has the advantage that, when $\gcd(d,q-1)>1$ and we restrict to a certain pencil of curves, it is optimal up to a constant multiplicative factor; see the end of Section~\ref{sec:alg} for discussion.

\section{Constructions from conics}\label{sec:conics}

In this section, we prove Theorem~\ref{thm:conics}. Throughout the section, we assume that $q$ is an odd prime power.

A key ingredient in our proof is an effective version of the Lang--Weil bound \cite{LW54}. A standard application of Weil's bound gives an asymptotic formula for the number of $x\in \F_q$ such that $f_i(x)$ is a square in $\F_q$ for all $i$, where $f_1, f_2, \ldots, f_\ell \in \F_q[x]$ are ``independent"; see, for example, \cite[Lemma 1]{Szo92a}. Recently, Slavov \cite{S26} extended this result to multivariable polynomials with the help of an explicit version of the Lang--Weil bound by Cafure and Matera~\cite{CM06}. The following lemma is a special case of his result \cite[Theorem 3 and Remark 13]{S26}.

\begin{lemma}[Slavov]\label{lem:S_LW}
Let $n,\ell,d$ be positive integers and $q$ be an odd prime power. Let $f_1, f_2,\ldots, f_\ell$ be polynomials in $\F_q[x_1,x_2,\ldots, x_n]$ with degree $d$. Suppose that for any nonempty subset $I \subseteq \{1,2,\ldots, \ell\}$, the product $\prod_{i \in I} f_i$ is not a constant multiple of the square of a polynomial in $\F_q[x_1,\ldots, x_n]$. Then the number of $(a_1, a_2, \ldots, a_n)\in \F_q^n$ such that $f_i(a_1,a_2, \ldots, a_n)$ is a nonzero square in $\F_q$ for all $1\leq i \leq \ell$ is 
$$
\frac{q^n}{2^\ell}+O((2d)^{2\ell} q^{n-1/2}+(2d)^{13\ell/3}q^{n-1}),
$$
where the implied constant in the error term is absolute.
\end{lemma}

Our proof also uses the concept of dual curves. Recall that the points of the \emph{dual projective plane} $(\mathbb{P}^2)^{\ast}$ correspond to the lines in $\mathbb{P}^2$. Given a geometrically irreducible plane curve $C$, the \emph{dual curve} $C^{\ast} \subset (\mathbb{P}^{2})^{\ast}$ parametrizes tangent lines to $C$. The set of $\mathbb{F}_q$-points on this curve, $C^{\ast}(\mathbb{F}_q)$, therefore represents those $\mathbb{F}_q$-lines that are tangent to $C$ at some point $P\in C(\overline{\mathbb{F}_q})$. 

The proof of Theorem~\ref{thm:conics} combines these two tools.

\begin{proof}[Proof of Theorem~\ref{thm:conics}] Suppose $\{C_1, C_2, \ldots, C_{\ell}\}$ is a collection of geometrically irreducible conics that block all lines in $\bP^2(\F_q)$. In particular, the union of these conics intersects with lines of the form $bx+cy-z=0$ with $b,c\in \F_q$. The following claim provides a simple criterion for determining whether such a line is skew to a given conic.
\begin{claim}\label{claim:nonsquare}
Let $C$ be a geometrically irreducible conic defined over $\F_q$:
$$C:a_{200}x^{2}+a_{020}y^{2}+a_{002}z^{2} + a_{110}xy+a_{101}xz+a_{011}yz=0.$$
Consider the polynomial $D(\alpha,\beta)\in \F_q[\alpha,\beta]$ defined by:
    $$
    D=(2a_{002}\alpha \beta+a_{110}+a_{101}\beta+a_{011}\alpha)^{2} 
-4(a_{200}+a_{002} \alpha^{2}+a_{101}\alpha)(a_{020}+a_{002}\beta^{2}+a_{011}\beta).
    $$
Then for each $b,c\in \F_q$, $C$ does not intersect the line $L\colon bx+cy-z=0$ at an $\F_q$-point provided the following two conditions hold:
\begin{enumerate}
    \item $(a_{200}+a_{002}b^{2}+a_{101}b)(a_{020}+a_{002}c^{2}+a_{011}c)\neq 0$.
    \item $D(b,c)$ is a non-square in $\F_q$.
\end{enumerate}
In particular, for all but at most $4q$ pairs $(b,c)\in \F_q \times \F_q$, the line $L\colon bx+cy-z=0$ is skew to $C$ if $D(b,c)$ is a non-square in $\F_q$.
\end{claim}
\begin{poc}
First, note that the equation $a_{200}+a_{002}\alpha^{2}+a_{101}\alpha=0$ has at most two solutions in $\F_q$. Otherwise, the polynomial in $\alpha$ would be identically zero, so  $a_{200}=a_{002}=a_{101}=0$; in this case $C$ is reducible with a factor of $y$. Similarly, the equation $a_{020}+a_{002}\beta^{2}+a_{011}\beta=0$ has at most two solutions in $\F_q$. Thus, condition (1) holds for all but at most $4q$ pairs $(b,c)\in \F_q \times \F_q$. 

Next, we compute $C \cap L$. Substituting $z=bx+cy$ into the equation of $C$ gives: 
$$a_{200}x^{2}+a_{020}y^{2}+a_{002}(bx+cy)^{2} +a_{110}xy+a_{101}x(bx+cy)+a_{011}y(bx+cy)=0,$$
which simplifies to
\begin{equation}\label{eq:intersection}
(a_{200}+a_{002} b^{2}+a_{101}b)x^{2} +(2a_{002}bc+a_{110}+a_{101}c+a_{011}b)xy +(a_{020}+a_{002}c^{2}+a_{011}c)y^{2}=0.    
\end{equation}
If $(a_{200}+a_{002}b^{2}+a_{101}b)(a_{020}+a_{002}c^{2}+a_{011}c)\neq 0$, then equation~\eqref{eq:intersection} has a nonzero solution over $\F_q$ only when its discriminant 
\begin{align*}
(2a_{002}bc+a_{110}+a_{101}c+a_{011}b)^{2}-4(a_{200}+a_{002}b^{2}+a_{101}b)(a_{020}+a_{002}c^{2}+a_{011}c)
\end{align*} 
is a square in $\F_q$. This proves the claim.
\end{poc}

For each $1\leq i \leq \ell$, let $D_i(\alpha,\beta)$ be the polynomial associated with the conic $C_i$ in Claim~\ref{claim:nonsquare}. We adopt the dual-coordinate convention in which a point $[t_0:t_1:t_2]\in(\mathbb P^2)^*$ corresponds to the line with equation $t_0x+t_1y-t_2z=0$ in $\mathbb{P}^2$. As a result, $[\alpha:\beta:1]$ corresponds to $\alpha x+\beta y-z=0$, and the calculation in Claim~\ref{claim:nonsquare} shows that $D_i(\alpha,\beta)=0$ precisely when this line is tangent to $C_i$. Consequently, $D_i=0$ is the restriction of the dual conic $C_i^*$ to the affine chart $t_2=1$. In particular, $D_i\in\F_q[\alpha,\beta]$ is an irreducible polynomial of degree $2$. Moreover, for $1\leq i<j\leq \ell$, $C_i$ and $C_j$ are distinct conics, so their dual curves are also distinct, that is, $C_{i}^{*}$ and $C_{j}^{*}$ are distinct; the corresponding polynomials $D_i$ and $D_j$ are distinct in the sense that $D_j \neq \lambda D_i$ for any $\lambda \in \F_q$. Therefore, for any nonempty subset $I \subseteq \{1,2,\ldots, \ell\}$, the polynomial $\prod_{i \in I} D_i$ is not a constant multiple of the square of a polynomial in $\F_q[\alpha,\beta]$. 

Fix a nonsquare $\nu\in\F_q^\times$. The polynomials $\nu D_1,\dots,\nu D_\ell$ satisfy the hypothesis of Lemma~\ref{lem:S_LW}, since multiplication by nonzero constants does not affect that hypothesis. Moreover, $\nu D_i(b,c)$ is a nonzero square if and only if $D_i(b,c)$ is a nonsquare. Applying Lemma~\ref{lem:S_LW} to $\nu D_1,\dots,\nu D_\ell$, we find that the number of pairs $(b,c)$ for which every $D_i(b,c)$ is a nonsquare is at least
$$
\frac{q^2}{2^\ell}-K (4^{2\ell} q^{3/2}+4^{13\ell/3}q),
$$
where $K$ is an absolute constant. It follows from Claim~\ref{claim:nonsquare} that the number of pairs $(b,c)\in \F_q \times \F_q$ for which the line $L\colon bx+cy-z=0$ is simultaneously skew to the conics $C_1,C_2, \ldots, C_{\ell}$ is at least
$$
\frac{q^2}{2^\ell}-K(4^{2\ell} q^{3/2}+4^{13\ell/3}q)- 4q\ell;
$$
however, by the blocking set assumption, no such line exists. It follows that
\begin{equation}\label{eq:bound-on-ell}
\frac{q^2}{2^\ell}-K(4^{2\ell} q^{3/2}+4^{13\ell/3}q)- 4q\ell \leq 0.
\end{equation}
We claim that \eqref{eq:bound-on-ell} implies that $\ell \geq c\log q$ for some absolute positive constant $c$. To see this, set $c=1/(20\log 2)$ and suppose that $\ell<c\log q$. Then $2^\ell<q^{1/20}$, and we obtain
\[
\frac{q^2}{2^\ell}>q^{39/20},\qquad
4^{2\ell}q^{3/2}<q^{17/10},\qquad
4^{13\ell/3}q<q^{43/30}.
\]
Also, $4q\ell=O(q\log q)$. These inequalities now show that the left-hand side of \eqref{eq:bound-on-ell} is positive for all sufficiently large $q$, a contradiction. Therefore, $\ell\geq c\log q$ for all sufficiently large $q$; after decreasing $c$ if necessary, the same conclusion holds for every odd prime power $q$.
\end{proof}

\begin{remark}\label{rem:S}
In the final remark of \cite{Szo92b}, Sz\H{o}nyi suggested a proof along the following lines, which shares some similarities with our proof. We now explain the details implicit in his remark. We fix a point $P\in\mathbb{P}^2(\mathbb{F}_q)$ (to be specified) and consider all the $q+1$ $\mathbb{F}_q$-lines in $\mathbb{P}^2$ that pass through $P$. Call these lines $L_1, L_2, \dots, L_{q+1}$. The condition that $L_i$ is skew to a given conic can be expressed as a certain single-variable quadratic function achieving a nonsquare value (this step requires some verification). Using Weil's bound, one can show that if the collection contains fewer than $c_0\log q$ conics, then at least one line $L_i$ through $P$ is skew to all of them. To apply Weil's bound, one must verify that the product of the polynomials in every nonempty subcollection is not a constant multiple of the square of a polynomial. To rule out this scenario, one needs to find a point $P$ such that none of the $q+1$  lines through $P$ is tangent to more than one conic. The difference between Sz\H{o}nyi's method and our proof is that we need not reduce to a single-variable polynomial, as we can rely on Lemma~\ref{lem:S_LW}.
\end{remark}

\section{Blocking families and Chebotarev Density Theorem}\label{sec:CDT}

This section is devoted to proving Theorem~\ref{thm:intro-unbounded}. Throughout the section, we assume that $q$ is odd. We establish a more general result (Theorem~\ref{thm:bounded-number-of-curves} below), showing that under a mild hypothesis, any bounded collection of curves fails to form a blocking set for sufficiently large $q$. Theorem~\ref{thm:intro-unbounded} will then follow as a corollary. 

To state Theorem~\ref{thm:bounded-number-of-curves}, we need a definition. A plane curve $C$ is called \emph{reflexive} if a generic tangent line to $C$ has contact of order exactly $2$ (i.e., is not a flex) and is tangent at a unique point (i.e., not bitangent). Every geometrically irreducible plane curve of degree $d$ is reflexive when the characteristic $p$ satisfies $p>d$ (see \cite[p.~5]{Hef89}). This motivates the hypothesis in Theorem~\ref{thm:intro-unbounded}.

\begin{theorem}\label{thm:bounded-number-of-curves}
Let $C\subset \mathbb{P}^2$ be a plane curve of degree $\widetilde{d}$ defined over $\mathbb{F}_q$. Suppose each geometrically irreducible component of $C$ has degree at least 2 and is reflexive. Then $C(\F_q)$ is not a blocking set in $\bP^2(\F_q)$ for $q$ sufficiently large with respect to $\widetilde{d}$. 
\end{theorem}

We use the notation $\widetilde{d}$ for the total degree of the curve $C$; we reserve the letter $d$ for the degree of a geometrically irreducible curve in this section.

To prove that $C(\mathbb{F}_q)$ is not a blocking set, it suffices to demonstrate the existence of at least one $\mathbb{F}_q$-line $L$ that is skew to $C$, meaning $(L \cap C)(\mathbb{F}_q) = \emptyset$. The existence of such a skew line is an arithmetic question over $\mathbb{F}_q$, but it can be studied by analyzing the geometry of the intersection over the algebraic closure $\overline{\F_q}$.

To build this connection, let us first consider the case where $C$ is a single geometrically irreducible curve of degree $d$. For a transverse $\mathbb{F}_q$-line $L$, the intersection $L \cap C$ consists of $d$ distinct points in $\mathbb{P}^2(\overline{\F_q})$. The geometric Frobenius map, $\sigma\colon [x:y:z]\mapsto [x^q:y^q:z^q]$, permutes these $d$ points because both $L$ and $C$ are defined over $\mathbb{F}_q$. This permutation partitions the set of intersection points into orbits. Each orbit corresponds to a set of roots of an irreducible polynomial over $\mathbb{F}_q$, and the size of the orbit is the degree of that polynomial. The partition of the integer $d$ into the sizes of these orbits is the cycle type of the permutation, which defines a unique conjugacy class in the symmetric group $S_d$. We denote this class by $\operatorname{Frob}(C \cap L)$.

This framework extends naturally to a reducible curve whose $\mathbb{F}_q$-irreducible components are geometrically irreducible. Let $C = \bigcup_{i=1}^m C_i$ be a plane curve with geometrically irreducible components $C_i$ of degree $d_i$. For a transverse $\mathbb{F}_q$-line $L$, the Frobenius map again permutes the intersection points. Since each $C_i$ is defined over $\mathbb{F}_q$, the action preserves the subsets $L \cap C_i$. We can therefore analyze the permutation on each subset independently. The action on the $d_i$ points of $L \cap C_i$ defines a conjugacy class in $S_{d_i}$ as described above. Taken together, the total permutation defines a conjugacy class in the product group $S_{d_1}\times \cdots \times S_{d_m}$.

Crucially, a point in the intersection $C\cap L$ is defined over $\mathbb{F}_q$ if and only if it is a fixed point of the Frobenius permutation. Therefore, a line $L$ is skew to $C$ if and only if its associated Frobenius action is a \emph{derangement} (a permutation with no fixed points). Our task is now translated into an arithmetic-geometric one: counting lines whose Frobenius action corresponds to a derangement.

To count these lines, we use a version of the Chebotarev density theorem, due to Entin \cite[Theorem 1]{Ent21}. Entin's theorem is stated for the slightly more general condition of \emph{quasireflexivity} to handle cases in characteristic $2$, but since we assume $q$ is odd in this section, the reflexivity condition is sufficient for our purposes. Indeed, when $\operatorname{char}(\mathbb{F}_q)$ is odd, the notions of reflexivity and quasireflexivity are equivalent \cite[Proposition 2.1]{Ent21}.

\begin{theorem}[Entin]\label{thm:entin}
Let $C\subset \mathbb{P}^2$ be a reflexive plane curve of degree $\tilde{d}$ defined over $\mathbb{F}_q$. Suppose the irreducible components of $C$ are $C_1, \dots, C_m$ where each $C_i$ is geometrically irreducible and $\deg(C_i)=d_i$ for $1\leq i\leq m$. Let $U\subseteq(\mathbb{P}^2)^{\ast}$ denote the open subset of lines not tangent to $C$. Let $\mathcal{C}$ be a conjugacy class in the product group $S_{d_1}\times \cdots \times S_{d_m}$. Then
\[
|\{L\in U(\mathbb{F}_q) : \operatorname{Frob}(C\cap L)=\mathcal{C}\}| = \frac{|\mathcal{C}|}{|S_{d_1}\times S_{d_2}\times\cdots \times S_{d_m}|} q^2 \left(1+O_{\widetilde{d}}(q^{-1/2})\right).
\]
\end{theorem}

We illustrate Theorem~\ref{thm:entin} in the special case where the conjugacy class in $S_{d_1}\times \cdots \times S_{d_m}$ corresponds to special derangements. More precisely, let $\mathcal{C}$ be the conjugacy class of derangements corresponding to a product of cycles of full length, i.e., a $d_i$-cycle in each component $S_{d_i}$. The size of this conjugacy class is 
$
|\mathcal{C}| = \prod_{i=1}^{m}(d_i-1)!.
$
Applying Theorem~\ref{thm:entin}, the number of transverse $\mathbb{F}_q$-lines $L$ for which $\operatorname{Frob}(C\cap L)=\mathcal{C}$ is 
\begin{equation}\label{eq:lbC}
\frac{|\mathcal{C}|}{d_1! \cdot d_2! \cdots d_{m}!} q^2 \left(1+O_{\widetilde{d}}(q^{-1/2})\right) = \left(\prod_{i=1}^{m} \frac{1}{d_i}\right) q^2 \left(1+O_{\widetilde{d}}(q^{-1/2})\right).
\end{equation}
Since this quantity is positive for sufficiently large $q$, there must exist at least one line that is skew to all geometrically irreducible components $C_1, \dots, C_m$. In particular, for $q$ sufficiently large, we see that a positive fraction of $\mathbb{F}_q$-lines are skew to $C$. We now have the tools to prove the main technical result of this section.

\begin{proof}[Proof of Theorem~\ref{thm:bounded-number-of-curves}] Decompose $C$ into its irreducible components over $\mathbb{F}_q$ in the following way:
\[
C = C_1 \cup \cdots \cup C_{m} \cup C_{m+1} \cup \cdots \cup C_{r},
\]
where
\begin{itemize}
    \item $C_i$ is irreducible over $\mathbb{F}_q$ and $d_i=\deg(C_i)\geq 2$ for each $i\geq 1$.
    \item $C_i$ is irreducible over $\overline{\mathbb{F}_q}$ for each $i\leq m$, and $C_i$ is \emph{not} irreducible over $\overline{\mathbb{F}_q}$ for each $i>m$. 
\end{itemize}
By hypothesis, $C_i$ is reflexive for each $1\leq i\leq m$. By the above discussion, the number of $\mathbb{F}_q$-lines that are skew to $C_i$ for each $1\leq i\leq m$ is at least the quantity given by equation~\eqref{eq:lbC}. On the other hand, for each $m+1\leq i\leq r$, the curve $C_i$ is irreducible over $\mathbb{F}_q$ but not geometrically irreducible. For these curves, a standard application of B\'ezout's theorem shows that $|C_i(\mathbb{F}_q)|\leq \frac{d_i^2}{4}$ (see, for example, \cite[Lemma 3.1]{AGY23}). Consequently, for each $m+1\leq i\leq r$, all but at most $\frac{d_i^2}{4}(q+1)$ lines are skew to the curve $C_i$. Thus, the number of skew lines to $C$ is at least:
\begin{align}
\left(\prod_{i=1}^{m} \frac{1}{d_i}\right) q^2 \left(1+O_{\widetilde{d}}(q^{-1/2})\right) - \frac{q+1}{4}\sum_{i=m+1}^{r} d_i^2 \geq \frac{1}{\widetilde{d}^m} q^2 \left(1+O_{\widetilde{d}}(q^{-1/2})\right) - \frac{(q+1)\widetilde{d}^2}{4}.\label{eq:lower-bound-skew}
\end{align}
For $q$ sufficiently large with respect to $\widetilde{d}$, the lower bound \eqref{eq:lower-bound-skew} yields a positive fraction of $\mathbb{F}_q$-lines $L$ for which $(L\cap C)(\mathbb{F}_q)=\emptyset$. In particular, $C(\mathbb{F}_q)$ is not a blocking set. 
\end{proof}

We are now equipped to prove Theorem~\ref{thm:intro-unbounded}.

\begin{proof}[Proof of Theorem~\ref{thm:intro-unbounded}]
The result is a direct consequence of Theorem~\ref{thm:bounded-number-of-curves}. If $\ell(q)$ were bounded by a constant, then the corresponding union of curves would have a total degree bounded by a constant. The condition $q \in \mathcal{Q}_d$ ensures that each component curve is reflexive, so for $q$ sufficiently large, Theorem~\ref{thm:bounded-number-of-curves} implies that this union cannot be a blocking set, a contradiction.
\end{proof}

\begin{remark}
The hypothesis $p>d$ in Theorem~\ref{thm:intro-unbounded} ensures that the component curves are reflexive. This condition $p>d$ can be relaxed for families of \emph{nonsingular} curves. By a result of Pardini~\cite{Par86}, a nonsingular curve of degree $d$ is nonreflexive only if $d \equiv 1 \pmod p$. Thus, the weaker condition $p \nmid (d-1)$ is sufficient to guarantee reflexivity for nonsingular curves. The conclusion of Theorem~\ref{thm:intro-unbounded} therefore holds if each curve $C_i$ is nonsingular and its degree $d_i$ satisfies $p \nmid (d_i-1)$.
\end{remark}

\section{First Proof of Theorem~\ref{thm:main}: randomized construction}\label{sec:random}
In this section, we give a non-constructive proof of Theorem~\ref{thm:main}. We rely on a covering lemma due to S. K. Stein~\cite{Ste74}, stated as in \cite[Lemma 2.3]{GS08}, whose proof uses a randomized construction.

\begin{lemma}[Stein]\label{lem:Stein}
Consider a bipartite graph with bipartition $A\cup B$. Let $\delta$ be the minimum degree of a vertex in $A$. If $|A|\geq 2$, then there is a set $B'\subseteq B$ such that
\begin{equation}\label{ineq:covering-bound}
|B'| \leq \Big\lceil |B| \frac{\log|A|}{\delta}\Big\rceil 
\end{equation}
and $B'$ dominates $A$ (that is, for each $a\in A$, there is $b\in B'$ such that $a$ and $b$ are adjacent).
\end{lemma}

As preparation, we need a few lemmas. The first lemma is about an estimate on binomial coefficients.

\begin{lemma}\label{lem:convex}
Let $d\geq 2$. If $D\mid d$, then we have $D \binom{d/D+2}{2}\leq \binom{d+2}{2}$.
\end{lemma}
\begin{proof}
Observe that the function $\binom{t+2}{2}$ is strictly convex for real $t>0$. It follows that for all positive integers $n$ and $m$, we have $\binom{n+m+2}{2}-\binom{n+2}{2}> \binom{m+2}{2}-\binom{0+2}{2}=\binom{m+2}{2}-1$, that is, $\binom{n+m+2}{2}\geq \binom{n+2}{2}+\binom{m+2}{2}$. Repeatedly applying this inequality, we obtain $D \binom{d/D+2}{2}\leq \binom{d+2}{2}$.
\end{proof}

The next lemma is a standard interpolation lemma; see, for example, \cite[Proposition 3.1]{AGY22b}. Let $V_d$ be the vector space of degree-$d$ homogeneous forms over $\F_q$, and write $N=\dim_{\F_q}V_d=\binom{d+2}{2}$. The independent conditions in the next lemma are the equations $F(P_i)=0$, viewed as linear conditions on the coefficients of $F\in V_d$. Their independence means that the forms passing through $P_1,\ldots,P_k$ constitute a vector subspace of $V_d$ of dimension $N-k$.

\begin{lemma}\label{lem:independence}
Fix a finite field $\mathbb{F}_q$, and consider any $k$ distinct $\F_q$-points $P_1, P_2, \ldots, P_k$ in $\mathbb{P}^2$. If $d\geq k-1$, then passing through $P_1, P_2, \ldots, P_k$ imposes linearly independent conditions in the parameter space of degree $d$ plane curves over $\F_q$. 
\end{lemma}

Before deducing the next corollary, we explain a piece of notation that will be used in the proof below. If $G\in \F_{q^D}[x,y,z]$, write
$G^{(q^i)}$ for the polynomial obtained by applying the $q^i$-power
Frobenius map to the coefficients of $G$, while leaving the variables fixed. We define the polynomial norm by
\[
\operatorname{Norm}_{\F_{q^D}/\F_q}(G)
   :=\prod_{i=0}^{D-1} G^{(q^i)}.
\]
Thus, $\operatorname{Norm}_{\F_{q^D}/\F_q}(G)\in \F_q[x,y,z]$; if $G$
is homogeneous of degree $e$, then its norm is homogeneous of degree
$De$.

\begin{corollary}\label{cor:independence}
Let $d\geq 3$ and $N=\binom{d+2}{2}$. Uniformly for all $P \in \mathbb{P}^2(\F_q)$, the number of geometrically 
irreducible degree $d$ plane curves defined over $\F_q$ that pass through $P$ is 
$\frac{q^{N-1}-O_{d}(q^{N-2})}{q-1}$. 
\end{corollary}
\begin{proof}
Fix a point $P\in \bP^2(\F_q)$. For each $1\leq j \leq d$, let $\mathcal{S}_j$ denote the set of degree $j$ homogeneous polynomials in $\F_q[x,y,z]$ (together with the zero polynomial) and let $\mathcal{T}_j \subseteq \mathcal{S}_j$ denote the subset of polynomials $F$ in $\mathcal{S}_j$ such that the curve $\{F=0\}$ passes through $P$. Note that for each $1\leq j \leq d$, we have $|\mathcal{S}_j|=q^{\binom{j+2}{2}}$ and  $|\mathcal{T}_j|=q^{\binom{j+2}{2}-1}$ by Lemma~\ref{lem:independence}. 

Let $\mathcal{R}_{d}\subseteq \mathcal{T}_{d}\setminus\{0\}$ denote the set of polynomials in $\mathcal{T}_d \setminus \{0\}$ that are reducible over $\F_q$. If $F\in \mathcal{R}_d$, then we can write $F=GH$ for some nonconstant polynomials $G, H$ such that the curve $\{G=0\}$ passes through $P$. It follows that
$$
|\mathcal{R}_d|\leq \sum_{j=1}^{d-1} |\mathcal{T}_j||\mathcal{S}_{d-j}|=\sum_{j=1}^{d-1} q^{\binom{j+2}{2}+\binom{d-j+2}{2}-1}\leq 2 \sum_{j=1}^{\lceil (d-1)/2 \rceil} q^{\binom{j+2}{2}+\binom{d-j+2}{2}-1}
$$
Since the function $\binom{t+2}{2}$ is strictly convex for real $t>0$, for each $1\leq j \leq d-2$, we have 
$\binom{d+2}{2}-\binom{d-j+2}{2}\geq \binom{j+2}{2}-\binom{0+2}{2}+d-j$, which implies that $N=\binom{d+2}{2}\geq \binom{j+2}{2}+\binom{d-j+2}{2}+1$. Since $d-2\geq \lceil (d-1)/2 \rceil$ for $d\geq 3$, it follows that
$$
|\mathcal{R}_d|\leq 2 \sum_{j=1}^{\lceil (d-1)/2 \rceil} q^{\binom{j+2}{2}+\binom{d-j+2}{2}-1} \leq d q^{N-2}.
$$
Let $\mathcal{G}_{d}\subseteq \mathcal{T}_{d}\setminus\{0\}$ denote the set of polynomials in $\mathcal{T}_d \setminus \{0\}$ that are irreducible over $\F_q$ but geometrically reducible. Note that if $F\in \mathcal{G}_d$, 
then necessarily $F=\operatorname{Norm}_{\F_{q^D}/\F_{q}} (G)$ for some $D \mid d$ with $D\geq 2$ and some homogeneous polynomial $G$ with degree $d/D$ defined over $\F_{q^D}$ such that the curve $\{G=0\}$ passes through $P$. It follows from Lemma~\ref{lem:convex} and Lemma~\ref{lem:independence} that
$$
|\mathcal{G}_d|\leq \sum_{D \mid d, D\geq 2} (q^D)^{\binom{d/D+2}{2}-1} \leq q^{-2} \sum_{D \mid d, D\geq 2} q^{D\binom{d/D+2}{2}} \leq q^{-2} \sum_{D \mid d, D\geq 2} q^{\binom{d+2}{2}}\leq (d-1) q^{N-2}.
$$

Combining the two estimates above, the corollary follows.
\end{proof}

We are now ready to present our first proof of Theorem~\ref{thm:main}. The proof shows that $c_d$ can be taken arbitrarily close to $4$ when $q$ is sufficiently large.

\begin{proof}[Proof of Theorem~\ref{thm:main}]
We build a bipartite graph with bipartition $A\cup B$ as in Lemma~\ref{lem:Stein}, where $A$ is the set of all $q^2+q+1$ lines in $\mathbb{P}^2$ defined over $\mathbb{F}_q$ and $B$ is the set of all geometrically irreducible curves defined over $\mathbb{F}_q$ with degree $d$. We draw an edge between a vertex $L\in A$ and a vertex $C\in B$ if the intersection $C\cap L$ contains an $\mathbb{F}_q$-point. 

Next, we give a lower bound on the minimum degree $\delta$ of a vertex in $A$. By definition, we fix an $\mathbb{F}_q$-line $L$ and count the number of geometrically irreducible curves $C\in B$ such that $(C\cap L)(\mathbb{F}_q)\neq \emptyset$. As each $\mathbb{F}_q$-line has $q+1$ points, we can express $L(\mathbb{F}_q)=\{P_1, P_2, \dots, P_{q+1}\}$. For each subset  $S\subseteq \mathbb{P}^2(\F_q)$, define
\begin{align*}
\psi(S) &= \# \{ \text{geometrically irreducible curves } C \text { of degree } d \text{ such that } S\subseteq C(\mathbb{F}_q)\}. 
\end{align*}
By the principle of inclusion-exclusion, the degree of the vertex $L$ in the bipartite graph is at least:
\begin{equation}\label{ineq:lower-bound-psi}
\sum_{1\leq i\leq q+1} \psi(\{P_i\}) - \sum_{1\leq i<j\leq q+1} \psi(\{P_i, P_j\}).
\end{equation}
Let $N=\binom{d+2}{2}$ denote the dimension of the $\mathbb{F}_q$-vector space parameterizing all degree $d$ homogeneous polynomials in three variables. By Lemma~\ref{lem:independence} and Corollary~\ref{cor:independence}, we have
$$
\psi(\{P\}) = \frac{q^{N-1}-O_d(q^{N-2})}{q-1},  
\quad \psi(\{P,Q\}) \leq \frac{q^{N-2}-1}{q-1}
$$
for any two distinct points $P,Q \in \bP^2(\F_q)$.
The lower bound \eqref{ineq:lower-bound-psi} for the degree of $L$ thus becomes:
\[
(q+1)\cdot\frac{q^{N-1}-O_d(q^{N-2})}{q-1} - \binom{q+1}{2} \cdot \frac{q^{N-2}-1}{q-1}=\frac{1}{2} q^{N-1} - O_d(q^{N-2}).
\]
This allows us to conclude that $\delta \geq \frac{1}{2} q^{N-1}-O_d(q^{N-2})$ for the minimum degree of a vertex in $A$. Applying inequality~\eqref{ineq:covering-bound}, we find a subset $B'\subseteq B$ dominating $A$ with
\begin{align*}
|B'| &\leq |B|\frac{\log(|A|)}{\delta} + 1 \leq \frac{q^{N}-1}{q-1}\cdot \frac{\log(q^2+q+1)}{\frac{1}{2} q^{N-1} - O_d(q^{N-2})} + 1 \\ 
&\leq \left(4+o(1)\right)\log q,
\end{align*}
as $q\to \infty$. Equivalently, we can find a blocking set constructed from $(4+o(1))\log q$ geometrically irreducible curves of degree $d$, as required.
\end{proof}

\begin{remark}\label{rem:multiple}
It is straightforward to modify the above proof to construct multiple blocking sets using a union of geometrically irreducible degree $d$ curves. Recall that for each positive integer $t$, a \emph{$t$-fold blocking set} in $\bP^2(\F_q)$ is a subset of $\bP^2(\F_q)$ such that it intersects each $\F_q$-line in at least $t$ points. To form a $t$-fold blocking set in $\bP^2(\F_q)$, a similar computation shows that $(\frac{2(t+1)!}{t}+o(1))\log q$ curves are sufficient if $d\geq \min \{t,3\}$. 
\end{remark}

\section{Second Proof of Theorem~\ref{thm:main}: explicit construction}\label{sec:alg}
Let $d\ge 3$ be an integer and consider the curves $C_\alpha$ (parametrized by $\alpha\in\F_q$) given by
\begin{equation}
\label{eq:1}
yz^{d-1}=x^d -\alpha z^d.
\end{equation}
Each $C_\alpha$ is geometrically irreducible: the defining equation~\eqref{eq:1} is linear in $y$, so any nontrivial factorization of $yz^{d-1}-x^d+\alpha z^d$ in $\overline{\F_q}[x,y,z]$ would involve a nonconstant factor from $\overline{\F_q}[x,z]$; however, this is impossible since $z^{d-1}$ and $x^d-\alpha z^d$ share no common factor in $\overline{\F_q}[x,z]$.

\begin{theorem}
\label{thm:construction}
There exists a subset $S\subseteq \F_q$ of size at most $1+\left\lfloor \frac{2\log q}{\log\left(\frac{d}{d-1}\right)}\right\rfloor$ with the property that 
\begin{equation}
\label{eq:2}
U:=\bigcup_{\alpha\in S} C_\alpha(\F_q)
\end{equation}
is a blocking set in $\bP^2(\F_q)$.
\end{theorem}

\begin{proof}
We will choose a subset $S\subseteq \mathbb{F}_q$ and set $U\colonequals \bigcup_{\alpha\in S} C_{\alpha}(\F_q)$. The goal is to select $S$ so that $U$ meets each $\F_q$-line of $\bP^2$. Our strategy is to use a greedy algorithm: we include $\F_q$-points of curves $C_\alpha$ for suitable values of $\alpha$, chosen sequentially to block as many new lines as possible at each step.

Let $L\subseteq \bP^2$ be an $\F_q$-line; so, the equation of $L$ is $ax+by+cz=0$ for some $[a:b:c]\in\bP^2(\F_q)$. Since $[0:1:0]\in C_\alpha$ for each $\alpha$, we may assume from now on that $b\ne 0$ (otherwise $[0:1:0]\in L(\F_q)$ already). Moreover, if $\alpha \in \F_q$ and $[x_0:y_0:z_0]\in C_{\alpha}(\F_q)$ with $z_0=0$, then necessarily $[x_0:y_0:z_0]=[0:1:0]$ and hence $[x_0:y_0:z_0] \notin L(\F_q)$. Thus, it suffices to compute $L(\F_q) \cap \{[x_0:y_0:z_0]\in \bP^2(\F_q): z_0\neq 0\}$. Writing $u\colonequals a/b$ and $v\colonequals c/b$, this set is given by $\{[x:-ux-v:1]: x \in \F_q\}$. Consequently, $U\cap L(\F_q)\neq \emptyset$ if and only if 
there is some $\alpha \in S$, such that there is $x\in \F_q$ with $[x:-ux-v:1]\in C_{\alpha}(\F_q)$, that is, $x^d+ux+v=\alpha$. 

To this end, for each $(u,v)\in \F_q\times \F_q$, define
\begin{equation}
\label{eq:7}
T_{u,v}\colonequals\left\{\gamma^d+u\gamma+v\colon \gamma\in\F_q\right\} \subseteq \F_q.
\end{equation}
By the discussion above, $U$ is a blocking set if and only if 
\begin{equation}\label{eq:6}
S \cap T_{u,v} \neq \emptyset \quad \text{for each } (u,v)\in\F_q\times \F_q.
\end{equation}
Since $f_{u,v}(x):=x^d+ux+v$ is a polynomial of degree $d$, every value $b\in\F_q$ has at most $d$ preimages in $\F_q$. Therefore, 
\begin{equation}
\label{eq:8}
\left|T_{u,v}\right|\ge \frac{q}{d}\quad \text{for each }(u,v)\in\F_q\times \F_q.
\end{equation}

We construct, by induction on $\ell$, a sequence of elements $\alpha_1,\dots,\alpha_\ell\in \F_q$ together with sets $S^{(i)}_{\alpha_i}\subseteq \F_q\times \F_q$ (for $i=1,\dots, \ell$) satisfying the following properties:
\begin{itemize}
\item[(I)]  for each $i=1,\dots, \ell$ and each $(u,v)\in S^{(i)}_{\alpha_i}$, we have $\alpha_i\in T_{u,v}$. Moreover, for each $(u,v)\in \F_q\times \F_q$, if $\alpha_i\in T_{u,v}$, then $(u,v)\in S^{(j)}_{\alpha_j}$ for some $1\leq j\le i$.
\item[(II)] the sets $S^{(i)}_{\alpha_i}$ are disjoint and 
\begin{equation}
\label{eq:42}
\sum_{i=1}^\ell \left|S^{(i)}_{\alpha_i}\right|\ge q^2\cdot \left(1-\left(\frac{d-1}{d}\right)^\ell\right). 
\end{equation}
\end{itemize}

We first prove the base case $\ell=1$ of the construction satisfying the conditions~(I)-(II). 
For each $\beta\in \F_q$, we define:
\begin{equation}
\label{eq:9}
S^{(1)}_\beta\colonequals\left\{(u,v)\in\F_q\times \F_q\colon \beta\in T_{u,v}\right\}.
\end{equation}
A simple counting argument, coupled with inequality~\eqref{eq:8}, yields:  
\begin{equation}
\label{eq:11}
\sum_{\beta\in\F_q}\left|S^{(1)}_{\beta}\right|=\sum_{(u,v)\in \F_q\times\F_q} \left|T_{u,v}\right|\ge q^2\cdot \frac{q}{d}.
\end{equation}
Choose $\alpha_1\in \F_q$ such that 
\begin{equation}
\label{eq:10}
\left|S^{(1)}_{\alpha_1}\right|\ge \left|S^{(1)}_\beta\right|\text{ for each }\beta\in \F_q.
\end{equation}
Inequalities~\eqref{eq:11}~and~\eqref{eq:10} yield:
\begin{equation}
\label{eq:12}
\left|S^{(1)}_{\alpha_1}\right|\ge \frac{q^2}{d}=q^2\cdot \left(1-\frac{d-1}{d}\right).
\end{equation}
This completes the proof of the base case $\ell=1$ for the construction of the points $\alpha_1,\dots, \alpha_\ell$ along with the sets $S^{(1)}_{\alpha_1},\dots, S^{(\ell)}_{\alpha_\ell}$ satisfying the properties~(I)-(II) above. 

We continue with the inductive step. Suppose we have constructed $\alpha_1,\dots, \alpha_k\in\F_q$ (for some  $k\ge 1$) along with some sets $S^{(1)}_{\alpha_1},\dots,S^{(k)}_{\alpha_k}\subseteq \F_q\times \F_q$ satisfying the properties~(I)-(II). We now construct another set $S^{(k+1)}_{\alpha_{k+1}}\subseteq \F_q\times \F_q$ corresponding to another point $\alpha_{k+1}\in\F_q$ still satisfying properties~(I)-(II).  In particular, the sets $S^{(1)}_{\alpha_1},\dots, S^{(k)}_{\alpha_k}$ are disjoint and
\begin{equation}
\label{eq:26}
\sum_{i=1}^k \left|S^{(i)}_{\alpha_i}\right| \ge q^2\cdot \left(1- \left(\frac{d-1}{d}\right)^k\right).
\end{equation}
By the inductive hypothesis~(I), we know that for each $i=1,\dots, k$ and each $(u,v)\in\F_q\times \F_q$, 
\begin{equation}
\label{eq:22}
\text{if $\alpha_i\in T_{u,v}$, then $(u,v)\in S^{(j)}_{\alpha_j}$ for some $j\le i$;} 
\end{equation}
also, the inductive hypothesis~(I) yields
\begin{equation}
\label{eq:44}
\text{for each $i=1,\dots, k$ and for each $(u,v)\in S^{(i)}_{\alpha_i}$, we have $\alpha_i\in T_{u,v}$.}
\end{equation}
We let 
\begin{equation}
\label{eq:24}
W\colonequals (\F_q\times \F_q)\setminus \left(\bigcup_{i=1}^k S^{(i)}_{\alpha_i}\right).
\end{equation} 
Also, we let $V\colonequals \F_q\setminus \{\alpha_1,\dots, \alpha_k\}$. 
Using equations~\eqref{eq:22}~and~\eqref{eq:24}, we have that 
\begin{equation}
\label{eq:23}
T_{u,v}\subseteq V\text{ for each }(u,v)\in W.
\end{equation} 
Then we define for each $\beta\in V$ the set:
$$S^{(k+1)}_\beta:=\left\{(u,v)\in W\colon \beta\in T_{u,v}\right\};$$
according to equations~\eqref{eq:23}~and~\eqref{eq:44}, we have that 
\begin{equation}
\label{eq:40}
S^{(k+1)}_\beta\cap S^{(i)}_{\alpha_i}=\emptyset\text{ for each }i=1,\dots, k.
\end{equation}
Due to the definition of each set $S^{(k+1)}_\beta$, we get that if $\beta\in T_{u,v}$ (for any $(u,v)\in \F_q\times \F_q$), then either $(u,v)\in S^{(i)}_{\alpha_i}$ for some $i=1,\dots, k$, or $(u,v)\in S^{(k+1)}_\beta$. Using again~\eqref{eq:23}, we obtain: 
\begin{equation}
\label{eq:27}
\sum_{\beta\in V}\left|S^{(k+1)}_\beta\right| = \sum_{(u,v)\in W} \left|T_{u,v}\right|\ge \left|W\right|\cdot \frac{q}{d}.
\end{equation}
In the last inequality from \eqref{eq:27}, we also employed~\eqref{eq:8}. 
Then we pick $\alpha_{k+1}\in V$ (clearly, $\alpha_{k+1}\ne \alpha_i$ for $i=1,\dots, k$ due to the definition of $V$) such that
\begin{equation}
\label{eq:28}
\left|S^{(k+1)}_{\alpha_{k+1}}\right|\ge \left|S^{(k+1)}_\beta\right|\text{ for each }\beta\in V.
\end{equation}
Therefore, equations~\eqref{eq:27} and \eqref{eq:28} yield
\begin{equation}
\label{eq:29}
\left|S^{(k+1)}_{\alpha_{k+1}}\right|\ge \frac{|W|\cdot q}{d\cdot |V|}> \frac{|W|}{d}.
\end{equation}
We note that the sets $S^{(i)}_{\alpha_i}$ are all disjoint for $i=1,\dots, k+1$ (by the inductive hypothesis coupled with equation~\eqref{eq:40}); furthermore, condition~(I) above is satisfied for the points $\alpha_1,\dots, \alpha_{k+1}$. Combining the  equations~\eqref{eq:29},~\eqref{eq:24}~and~\eqref{eq:26}, we obtain that
\begin{align*}
\sum_{i=1}^{k+1}\left| S^{(i)}_{\alpha_i}\right| 
&\ge \left|\bigcup_{i=1}^k S^{(i)}_{\alpha_i}\right| + \frac{q^2-\sum_{i=1}^k \left|S^{(i)}_{\alpha_i}\right|}{d}\\
&\ge \frac{q^2}{d}+ q^2\cdot \left(1-\left(\frac{d-1}{d}\right)^k\right)\cdot \frac{d-1}{d}\ge q^2\cdot \left(1-\left(\frac{d-1}{d}\right)^{k+1}\right),
\end{align*}
as desired for proving that also condition~(II) holds for $S^{(1)}_{\alpha_1},\dots, S^{(k+1)}_{\alpha_{k+1}}$.

So, inductively, we obtain the construction of points $\alpha_1,\dots, \alpha_\ell\in\F_q$ such that for the corresponding (disjoint) sets $S^{(i)}_{\alpha_i}\subseteq \F_q\times \F_q$, we have the inequality
\begin{equation}
\label{eq:31}
\left|\bigcup_{i=1}^\ell S^{(i)}_{\alpha_i}\right|\ge  q^2\cdot \left(1 - \left(\frac{d-1}{d}\right)^\ell\right).
\end{equation}
Furthermore, by construction, for each $(u,v)\in \bigcup_{i=1}^\ell S^{(i)}_{\alpha_i}$, there exists some $i\in\{1,\dots, \ell\}$ such that $\alpha_i\in T_{u,v}$. Our construction stops when we achieve that 
\begin{equation}
\label{eq:32}
\bigcup_{i=1}^\ell S^{(i)}_{\alpha_i}=\F_q\times\F_q 
\end{equation}
because then the corresponding set $S:=\{\alpha_1,\dots, \alpha_\ell\}$ will have the desired property~\eqref{eq:6}. So, in order to obtain~\eqref{eq:32}, it suffices to have that 
\begin{equation}
\label{eq:33}
\left|\bigcup_{i=1}^\ell S^{(i)}_{\alpha_i}\right|>q^2-1.
\end{equation}
Using inequality \eqref{eq:31}, we see that inequality~\eqref{eq:33} is achieved once we have:
\begin{equation}
\label{eq:34}
\left(\frac{d-1}{d}\right)^\ell <  \frac{1}{q^2}.
\end{equation}
So, indeed, we can find a set $S$ such that
$|S|\le 1+\left\lfloor \frac{2\log q}{\log\left(\frac{d}{d-1}\right)}\right\rfloor,$ as required. 
\end{proof}

\begin{remark}
The blocking set $U$ constructed in the above proof of Theorem~\ref{thm:construction} is nontrivial for all large $q$ (with respect to $d$). Indeed, using the geometric irreducibility of $C_{\alpha}$, an $\F_q$-line $L$ meets $C_{\alpha}$ in at most $d$ points by B\'ezout's theorem; in particular, $|U\cap L|\leq |S|\cdot d = O_d(\log q)$. As $|L(\F_q)|=q+1$, it follows that $U$ does not contain $L(\F_q)$ for any $\F_q$-line $L$, as desired.
\end{remark}

\begin{remark}\label{rem:alg}
Assume that $d'=\gcd(d,q-1)>1$. In this case, we can show that if $\bigcup_{\alpha\in S}C_\alpha(\F_q)$ is a blocking set in $\bP^2(\F_q)$, then necessarily $|S|\geq c_{d'}\log q$ for some constant $c_{d'}$ depending on $d'$. Indeed, by equation~\eqref{eq:6}, we have $S \cap T_{u,v}\neq \emptyset$ for all $u,v\in \F_q$; in particular, for each $v\in \F_q$, we have $S \cap T_{0,v}=S \cap \{a^{d'}+v: a \in \F_q\}\neq \emptyset$. Now, if $|S|<c_{d'}\log q$, then a standard application of Weil's bound (see for example \cite[Lemma 2.1]{AGY23}) shows that there is $x\in \F_q$, such that $s-x$ is not a $d'$-th power in $\F_q$ for each $s \in S$, that is,  $S \cap T_{0,x}=\emptyset$, contradicting the above assumption on $S$. Thus, $|S|\geq c_{d'}\log q$.
\end{remark}

We end the paper with the following open question regarding the family of curves $C_{\alpha}\subset \mathbb{P}^2$ defined by $yz^{d-1}=x^d-\alpha z^d$. 

\begin{question}\label{quest:final}
Given $q$ and $d$, what is the smallest possible size of $S \subset \F_q$ such that $\bigcup_{\alpha\in S}C_\alpha(\F_q)$  is a blocking set in $\bP^2(\F_q)$? 
\end{question}

When $\gcd(d,q-1)>1$, we have shown that the answer to Question~\ref{quest:final} is between $c_1\log q$ and $c_2\log q$, where $c_1$ is a constant depending only on $\gcd(d,q-1)$ and $c_2$ is a constant depending only on $d$. Getting an asymptotically sharp answer in this case seems challenging.

By contrast, if $\gcd(d,q-1)=1$, then our argument for the lower bound no longer applies. Indeed, the map $\mathbb{F}_q\to \mathbb{F}_q$ given by $x\mapsto x^{d}$ is a permutation, and so, it is not useful to consider $d$-th power residues. It would be interesting to establish the asymptotic behavior of the answer to Question~\ref{quest:final} in this case.

\section*{Acknowledgments}
We are grateful to Alexei Entin for helpful discussions regarding Theorem~\ref{thm:entin}. We also thank the anonymous referees for their valuable comments and suggestions. The second author was partially supported by a Discovery NSERC grant. The research of the third author was supported in part by an NSERC fellowship. 

\bibliographystyle{abbrv}
\bibliography{main}

\end{document}